\documentclass[12pt,twoside]{article}
\usepackage{amssymb}
\usepackage{amsmath}
\usepackage{amsfonts}
\usepackage{tikz}
\usetikzlibrary{arrows}

\usepackage{amsthm}
\usepackage{calc}

\usepackage{float}
\usepackage{newlfont}
\usetikzlibrary{calc}

\date{}

\begin{document}

\centerline{}

\centerline {\Large{\bf  On minimal ring extensions }}

\centerline{}

\centerline{\bf {Rahul Kumar\footnote{rahulkmr977@gmail.com} \& Atul
Gaur\footnote{gaursatul@gmail.com}}}

\centerline{Department of Mathematics}

\centerline{University of Delhi, Delhi, India.}

\centerline{}

\newtheorem{Theorem}{\quad Theorem}[section]

\newtheorem{Corollary}[Theorem]{\quad Corollary}

\newtheorem{Lemma}[Theorem]{\quad Lemma}

\newtheorem{Proposition}[Theorem]{\quad Proposition}

\theoremstyle{definition}

\newtheorem{Definition}[Theorem]{\quad Definition}

\newtheorem{Example}[Theorem]{\quad Example}

\newtheorem{Remark}[Theorem]{\quad Remark}

\begin{abstract} Let $R$ be a commutative ring with identity. The ring $R\times R$ can be viewed as an extension of $R$ via the diagonal map $\Delta: R \hookrightarrow R\times R$, given by $\Delta(r) = (r, r)$ for all $r\in R$. It is shown that, for any $a, b\in R$, the extension $\Delta(R)[(a,b)] \subset R\times R$ is a minimal ring extension if and only if the ideal $<a-b>$ is a maximal ideal of $R$. A complete classification of maximal subrings of $R(+)R$ is also given. The minimal ring extension of a von Neumann regular ring $R$ is either a von Neumann regular ring or the idealization $R(+)R/\mathfrak{m}$ where $\mathfrak{m}\in \text{Max}(R)$. If $R\subset T$ is a minimal ring extension and $T$ is an integral domain, then $(R:T) = 0$ if and only if $R$ is a field and $T$ is a minimal field extension of $R$, or $R_J$ is a valuation ring of altitude one and $T_{J}$ is its quotient field. 
\end{abstract}

13B25, 13F30\\


\section{Introduction}
All rings considered below are commutative with nonzero identity; all ring extensions, ring homomorphisms, and algebra homomorphisms are unital. For any ring $R$, let $\text{tq}(R)$ denotes the total quotient ring of $R$ and $\text{Max}(R)$ denotes the set of all maximal ideals of $R$. By an overring of $R$, we mean any subring of $\text{tq}(R)$ which contains $R$. For any ring extension $R\subseteq S$, the conductor $(R:S) := \{ s\in S~|~ sS\subseteq R \}$. By a local ring, we mean a ring with a unique maximal ideal. 

An injective ring homomorphism $f$ that is not an isomorphism is called a minimal ring homomorphism if any factorization $f=g\circ h$ entails that one of the ring homomorphisms $g,h$ is an isomorphism, see \cite{ferrand}. Let $R$ be any proper subring of a ring $T$. Then $T$ is called a minimal ring extension of $R$ or equivalently, $R$ is a maximal subring of $T$ if the inclusion map $R\hookrightarrow T$ is a minimal ring homomorphism, that is, if there is no ring $S$ such that $R\subset S\subset T$ where $\subset$ denotes proper inclusion. By a minimal overring of $R$, we mean any overring of $R$ which is a minimal ring extension of $R$. Note that if $R\subset T$ is a minimal ring extension, then either $R\subset T$ is an integral ring extension or $R\hookrightarrow T$ is a flat epimorphism, see \cite[Th\'{e}or\`{e}me~2.2]{ferrand}.
 
If $R$ is a ring, then $R$ can be viewed as a subring of $R\times R$ via the diagonal map, that is, via the canonical injective ring homomorphism, $\Delta: R \hookrightarrow R\times R$, given by $\Delta(r) = (r,r)$ for all $r\in R$. It was shown in \cite[Lemma~2.1]{dobbs1} that $\Delta(R)[\left(r,s\right)] = R\times R$ for $r, s \in R$ if and only if $r-s \in U(R)$, where $U(R)$ denote the set of units of $R$. Dobbs \cite[Proposition~2.2]{dobbs1} also proved that $\Delta(R) \subset R\times R$ is a minimal ring extension if and only if $R$ is a field. In Theorem \ref{t4}, we show that, for any $r, s\in R$, $\Delta(R)[\left(r,s\right)] \subset R\times R$ is a minimal ring extension if and only if the ideal $<r-s>$ is a maximal ideal of $R$.

If $R$ is a domain but not a field, then minimal ring extensions of $R$ are the $R$-algebras that are isomorphic to one of the following three types of rings: a minimal overring of $R$; an idealization $R\left(+\right) R/\mathfrak{m}$ where $\mathfrak{m}\in \text{Max}(R)$; a direct product $R\times R/\mathfrak{m}$ where $\mathfrak{m}\in \text{Max}(R)$, see \cite[Theorem~2.7]{dobbs2}. This result is generalized by assuming that $\text{tq}(R)$ is a von Neumann regular ring and $\text{Max}(R) \cap \text{Min}(R) = \phi$, see \cite[Corollary~2.5]{dobbs3}. Dobbs and Shapiro also classified the integral minimal ring extensions of $R$, see \cite[Proposition~2.12]{dobbs3}. In Theorem \ref{t1} and \ref{t2}, we classify the minimal ring extension of a von Neumann regular ring, and thereby settled the open problem posed by Dobbs in \cite[p. 35]{dobbs4}.

 Recall \cite[cf. Nagata, 1962, p.2]{nagata} that if $R$ is a ring and $E$ is an $R$-module, then the idealization $R(+)E$ is the ring defined as follows: Its additive structure is that of the abelian group $R \oplus E$, and its multiplication is defined by $\left(r_1,e_1\right) \left(r_2,e_2\right) := \left(r_1 r_2,r_1 e_2 + r_2 e_1\right)$ for all $ r_1, r_2\in R$ and $e_1, e_2 \in E$. It will be convenient to view $R$ as a subring  of $R(+)E$ via the canonical injective ring homomorphism that sends $r$ to $\left( r,0\right)$. Note that every ring has a minimal ring extension, see $\cite{dobbs1}$. However, $\mathbb{Z}$ has no maximal subring, that is, maximal subrings need not always exist. In Corollary $\ref{r1}$, we show that for any ring $R$, the ring $R(+)R$ has maximal subrings. In Proposition \ref{t5}, we prove that $R(+)Rb$ is a maximal subring of $R(+)R$ if and only if $Rb$ is a maximal ideal of $R$. 

Let $R\subset T$ be a minimal ring extension. By \cite[Th\'{e}or\`{e}me~2.2(i)]{ferrand} and \cite[Lemme~1.3]{ferrand}, there exists a unique maximal ideal $J$ of $R$ such that $R_J\hookrightarrow T_J:= T_{R\setminus J}$ is not an isomorphism; moreover, $R_J\hookrightarrow T_J$ is then a minimal ring extension, and $R_P\hookrightarrow T_P$ is an isomorphism for all $P\in \text{Spec}(R)\setminus \{J\}$. The maximal ideal $J$ appearing in the above statement is called the crucial maximal ideal \cite[Definition~2.9]{lucy}.

The Proposition $2.11$ of $\cite{lucy}$ states that if $R\subset T$ is a minimal ring extension, then the crucial maximal ideal is the only maximal ideal which contains $(R:T)$. In \cite[Corollary~2.14]{lucy}, the author states that if $R\subset T$ is a minimal ring extension and $T$ is an integral domain, then $(R:T)$ = 0 if and only if $R$ is a field and $T$ is a field extension of prime degree over $R$, or $R$ is a valuation ring of altitude one and $T$ is its quotient field. We give an example which shows the above mentioned proposition and corollary are not true.

\section{\emph{Maximal subrings of certain commutative rings}}

The problem of classifying the minimal ring extensions of a von Neumann regular ring was posed by Dobbs in \cite{dobbs4}. In our first theorem, we present a complete classification of minimal ring extensions of a von Neumann regular ring.

\begin{Theorem}\label{t1}
Let $R\subset T$ be a minimal ring extension where $R$ is a von Neumann regular ring. Then either $T$ is a von Neumann regular ring or $T\cong R (+) R/\mathfrak{m}$ (as $R$-algebra) for some maximal ideal $\mathfrak{m}$ of $R$.
\end{Theorem}

\begin{proof}
Since $R$ is von Neumann regular, $R$ is reduced. First assume that $T$ not reduced. Then by \cite[Proposition~2.3]{dobbs3}, $T\cong R(+) R/\mathfrak{m}$ (as $R$-algebra) for some maximal ideal $\mathfrak{m}$ of $R$. Now, assume that $T$ is a reduced ring. If $R\hookrightarrow T$ is a flat epimorphism, then by \cite[Proposition~3.9]{glaz}, $T$ is an overring of $R$. This is a contradiction as $\text{tq}(R) = R$. Thus, by \cite[Th\'{e}or\`{e}me~2.2]{ferrand}, $T$ is an integral extension of $R$ and hence $\text{dim}(T) = \text{dim}(R) = 0$. Therefore, $T$ is a von Neumann regular ring. 
\end{proof}

The next theorem further characterizes the minimal ring extensions of a von Neumann regular ring. 

\begin{Theorem}\label{t2}
Let $R$ be a von Neumann regular ring. Then $T$ is a minimal ring extension of $R$ if and only if there exists a maximal ideal $\mathfrak{m}$ of $R$ such that one of the following three conditions holds:
\begin{enumerate}

\item[(i)] $\mathfrak{m}$ is a maximal ideal of $T$ and $T/\mathfrak{m}$ is a minimal field extension of $R/\mathfrak{m}$;

\item[(ii)] There exists $q\in T\setminus R$ such that $T = R[q]$, $q^2-q\in \mathfrak{m}$, and $\mathfrak{m}q\subseteq R$;

\item[(iii)] There exists $q\in T\setminus R$ such that $T = R[q]$, $q^2\in R$, $q^3\in R$, and $\mathfrak{m}q\subseteq R$. 

If any of the above three conditions holds, then $\mathfrak{m}$ is uniquely determined as $(R:T)$. Also (i)-(iii) are mutually exclusive.
\end{enumerate}
\end{Theorem}

\begin{proof}
Since $R$ is a von Neumann regular, we have $\text{tq}(R) = R$. If $R\hookrightarrow T$ is a flat epimorphism, then $T$ is an overring of $R$, by \cite[Proposition~3.9]{glaz}, which is not possible. Thus, by \cite[Th\'{e}or\`{e}me~2.2]{ferrand}, any minimal ring extension of $R$ is an integral extension of $R$. Now, the result follows by \cite[Proposition~2.12]{dobbs3}.
\end{proof}

In \cite[Theorem~2.4]{dobbs3}, a characterization of minimal ring extension of a reduced ring $R$ such that the total quotient ring of $R$, is a von Neumann regular ring, is given. However, till now we do not know any minimal ring extension of a non-reduced ring $R$ other than $R(+)R/\mathfrak{m}$, where $\mathfrak{m}$ is a maximal ideal of $R$. In the next theorem, we have shown that $R\times R$ is a minimal ring extension of its subring which may not be reduced.

\begin{Theorem}\label{t4}
For any ring $R$, let $\Delta: R \hookrightarrow R\times R$ be the diagonal map, given by $\Delta(r) = \left(r,r\right)$ for all $r\in R$. Then for any $a, b\in R$, $\Delta(R)[(a,b)] \subset R\times R$ is a minimal ring extension if and only if the ideal $<a-b>$ is a maximal ideal of $R$.
\end{Theorem}

\begin{proof}
First, we claim that \begin{align}\label{eq1}
\Delta(R)[(a,b)] &=\{(c,d)\in R\times R~|~ c-d\in <a-b>\}.
\end{align}

Let $(c,d)\in R\times R$ such that $c-d\in <a-b>$. Then $c-d = (a-b)t$ for some $t\in R$. As
\begin{align*}
(c,d) &=  (c-ta,c-ta) + (t,t)(a,b),
\end{align*} 
\noindent
we conclude that $(c,d)\in \Delta(R)[(a,b)]$.
Now, assume that $(e,f)\in \Delta(R)[(a,b)]$. So,
$$(e,f) = (a_0,a_0) + (a_1,a_1)(a,b) + (a_2,a_2)(a,b)^2 + \cdots + (a_n,a_n)(a,b)^n,$$ where $(a_i, a_i) \in \Delta(R)$ for all $i$. 
This gives,
\begin{align}
e &= a_0 + a_1 a + a_2 a^2 + \cdots + a_n a^n, \label{3}\\
f &= a_0 + a_1 b + a_2 b^2 + \cdots + a_n b^n. \label{4}
\end{align}
On subtracting $(\ref{4})$ from $(\ref{3})$, we have

$$e-f = a_1 (a-b) + a_2(a^2-b^2) + \cdots + a_n(a^n-b^n).$$

This gives $e-f\in <a-b>$. So, the claim holds. Now, suppose that $<a-b>$ is a maximal ideal of $R$. We assert that $\Delta(R)[(a,b)] \subset R\times R$. If possible, suppose $\Delta(R)[(a,b)] = R\times R$. Then $(1,0) \in \Delta(R)[(a,b)]$. Therefore, by $(\ref{eq1})$, we have $1\in<a-b>$, which is a contradiction. Therefore, $\Delta(R)[(a,b)] \neq R\times R$. Now, to show that $\Delta(R)[(a,b)] \subset R\times R$ is a minimal ring extension, enough to show that $(\Delta(R)[(a,b)]) [(e,f)] = R\times R$ for any $(e,f) \in (R\times R)\setminus\Delta(R)[(a,b)]$.

Note that $e-f \notin <a-b>$, by $(\ref{eq1})$. Therefore, $<a-b> + <e-f> = R$ and hence  

$$1 = (a-b)t_1 + (e-f)t_2 ~ \mbox{for some}~ t_1, t_2\in R.$$

This gives,  

$$(1,0) = ((a-b) t_1,0) + ((e-f) t_2,0).$$

Now, by $(\ref{eq1})$, we have $$((a-b) t_1, 0) \in  \Delta(R)[(a,b)] \subseteq (\Delta(R)[(a,b)]) [(e,f)]$$ and $$((e-f) t_2, 0) \in  \Delta(R)[(e,f)] \subseteq (\Delta(R)[(a,b)]) [(e,f)].$$ Thus, $(1,0) \in (\Delta(R)[(a,b)])[(e,f)]$. Similarly, $(0,1) \in (\Delta(R)[(a,b)])[(e,f)]$ and hence the claim holds. 

Conversely, suppose that $\Delta(R)[(a,b)] \subset R\times R$ is a minimal ring extension. First we assert that $<a-b>$ is a proper ideal of $R$. If possible, suppose that $1\in <a-b>$. Then $(1,0), (0,1)\in \Delta(R)[(a,b)]$ by $(\ref{eq1})$. It follows that $\Delta(R)[(a,b)] = R\times R$, a contradiction. Thus, $<a-b>$ is a proper ideal of $R$. Now, let $I$ be any ideal of $R$ properly containing the ideal $<a-b>$. Choose $e\in I\setminus<a-b>$. Then by (\ref{eq1}), $\left(e,0\right) \notin \Delta(R)[\left(a,b\right)]$. By minimality, we conclude that $\left(\Delta(R)[\left(a,b\right)]\right) [\left(e,0\right)] = R\times R$. Thus,
$$(1,0) = (a_0,b_0) + (a_1,b_1)(e,0) + (a_2,b_2)(e,0)^2 + \cdots + (a_n,b_n)(e,0)^n,$$ where  $(a_i,b_i) \in \Delta(R)[(a,b)]$ for all $i$. 

This gives, $$1 = a_0 + a_{1}e + \cdots + a_{n}e^n~\mbox{and}~ b_0 = 0.$$ Now, by $(\ref{eq1})$, $a_0-b_0 \in <a-b> \subset I$. As $a_{i}e\in I$ for all $i$, we must have $1\in I$. Therefore, $<a-b>$ is a maximal ideal of $R$.
\end{proof} 

\begin{Remark}
Note that \cite[Proposition~2.2]{dobbs1} is a particular case of Theorem \ref{t4} with $a = b$. 
\end{Remark}
 
Note that a maximal subring of a ring $R$ may not exists. For example, the ring of integers $\mathbb{Z}$ does not admit any maximal subring. However, $R(+)R$ always admits a maximal subring as we have in the next result. In fact, in the next proposition, we present a complete classification of maximal subrings of $R(+)R$.

\begin{Proposition}\label{t5}
For any ring $R$, let $R \hookrightarrow R(+) R$ be the canonical injective ring homomorphism, given by $r \mapsto \left(r,0\right)$ for all $r\in R$. Then for any $a, b\in R$, $R[(a,b)] \subset R(+) R$ is a minimal ring extension if and only if the ideal $<b>$ is a maximal ideal of $R$.
\end{Proposition}

\begin{proof}
Note that $R[(a,b)] = R(+)<b>$, by \cite[Lemma~2.3]{dobbs1}. First suppose that $R[(a,b)] \subset R(+) R$ is a minimal ring extension. Thus, $<b>$ is a proper ideal of $R$. Let $I$ be any ideal of $R$ properly containing $<b>$. Then we have $R(+)<b> \subset R(+)I$. It follows that $R(+)I = R(+)R$ and so $I = R$. Therefore, $<b>$ is a maximal ideal of $R$.

Conversely, assume that $<b>$ is a maximal ideal of $R$. Thus, $R[(a,b)] \subset R(+) R$ as $R[(a,b)] = R(+)<b>$. Let $T$ be a subring of $R(+) R$ containing $R[(a,b)]$ properly. Then by \cite[Remark~2.9]{dobbs1}, $T = R(+)I$ for some ideal $I$ of $R$. It follows that $<b>\subset I$ and so $I = R$. Therefore, $R[(a,b)] \subset R(+) R$ is a minimal ring extension.
\end{proof}

The following corollaries can be deduced immediately from the above proposition.

\begin{Corollary}\label{r1}
Let $R$ be any ring and $M$ be a maximal ideal of $R$. Then $R(+)M$ is a maximal subring of $R(+)R$. In particular, $R(+)R$ has maximal subrings for any ring $R$.  
\end{Corollary}

\begin{Corollary}\label{r2}
Let $R$ be a ring. Then $R$ is a maximal subring (upto isomorphism) of $R(+)R$ if and only if $R$ is a field.
\end{Corollary}

We end this section with the following remark. 

\begin{Remark} In \cite[Corollary~2.8]{azarang1}\label{t6}, Azarang proved that every finitely generated algebra over a commutative ring has a maximal subring. The result does not seem to be correct as there are rings with no maximal subring. For example, the ring of integers $\mathbb{Z}$ does not admit any maximal subring. Clearly, any such ring is a finitely generated algebra over itself. 
\end{Remark}

\section{\bf Correction to some known results}

We assume throughout that $J$ denote the crucial maximal ideal of minimal ring extension $R\subset T$ unless otherwise stated. For completeness, we first list the results which we are going to discuss in this section. 

\begin{enumerate}

\item[(1)] \cite[Proposition~2.11]{lucy} Let $R\subset T$ be a minimal ring extension. Then $(R:T)\in \text{Spec}(R)$ and $J$ is the only maximal ideal in $R$ which contains $(R:T)$. Moreover, if no maximal ideal in $T$ lies over $J$, then the following statement holds: $(R:T)\subset J$, $T_{J} = R_{(R:T)}$ is local, $(R_J:T_{J}) = (R:T)R_J$ is the maximal ideal in $T_{J}$, $\text{height}(J/(R:T)) = 1$, and $(R:T)T\in \text{Max}(T)$.

\item[(2)] \cite[Corollary~2.14]{lucy} If $R\subset T$ is a minimal ring extension and $T$ is an integral domain, then $(R:T) = 0$ if and only if $R$ is a field and $T$ is a field extension of prime degree over $R$, or $R$ is a valuation ring of altitude one and $T$ is its quotient field.

\item[(3)] \cite[Proposition~3.2(3)]{glaz} Let $f:R\hookrightarrow T$ be a minimal ring homomorphism. If $f:R\hookrightarrow T$ is a flat epimorphism, then $R/(R:T)$ is a one-dimensional local domain, $(R:T)\in \text{Max}(T)$ and $T_J = R_{(R:T)}$.

\item[(4)] \cite[Proposition~3.5]{glaz} Let $R\hookrightarrow T$ be an injective ring homomorphism. Then $R\hookrightarrow T$ is minimal and a flat epimorphism if and only if $R/(R:T)$ is a one-dimensional valuation ring and $T/(R:T)$ is its quotient field.
\end{enumerate}

We now present a counter example to show that $(1)$ is not fully correct. More precisely, we show that $J$ may not be the only maximal ideal containing $(R:T)$ and $(R:T)T$ may not belongs to $\text{Max}(T)$. In fact, there may be infinitely many maximal ideals containing $(R:T)$. The example also proves that $(2)$ is completely incorrect. On page $310$ of $\cite{glaz1}$, the authors mentioned that the assumption of $R$ to be local in above results $(3)$ and $(4)$ is missing due to printing mistake. Our next example shows that why this extra assumption is needed in above results $(3)$ and $(4)$. Moreover, we prove the modified version of $(3)$ and $(4)$ (where we do not need $R$ to be local) in Proposition $\ref{p1}$ and Theorem $\ref{t9}$, respectively.

\begin{Example}\label{e1}
Let $R = \mathbb{Z}$, $T = \mathbb{Z}[1/2]$. We assert that $R\subset T$ is a minimal ring extension. Suppose there is a ring $S$ such that $R\subset S\subseteq T$. Choose $f(1/2) = \sum_{i=0}^{n} \alpha_{i}(1/2)^i\in S\setminus R$. Then $f(1/2) = m/2^k$ for some $k \in \mathbb{N}$ and $m\in R$. Thus, $m/2 = 2^{k-1} (m/2^k) \in S$, which gives 1/2 $\in$ $S$. Therefore, $T$ is a minimal ring extension of $R$. Note that $(R:T) = 0$, as for every $\alpha \in  R$, there exists $n \in \mathbb{N}$ such that $\alpha/2^n$ is not an integer. Now crucial maximal ideal $J$ of the extension $R\subset T$ is $2\mathbb{Z}$ as $R_J\hookrightarrow T_J$ is not an isomorphism and $R_P\hookrightarrow T_P$ is an isomorphism for all $P\in \text{Spec}(R)\setminus \{J\}$. This counters $(1)$ as every maximal ideal of $R$ contains $(R:T)$. Also $0 = (R:T)T \not\in \text{Max}(T)$. As $R$ is not a field and neither $R$ is a valuation ring nor $T$ is its quotient field, this counters $(2)$ completely. Now, observe that $R$ is integrally closed in $T$. So, Ferrand's dichotomy \cite[Th\'{e}or\`{e}me~2.2]{ferrand} gives that the inclusion map $f:R \hookrightarrow T$ is a flat epimorphism. This shows that the assumption of $R$ to be local is needed in $(3)$ and $(4)$.
\end{Example}

Though the above example shows that there may be infinitely many maximal ideals in $R$ containing $(R:T)$ and $(R:T)T$ may not belong to $\text{Max}(T)$, however, the remaining statement of \cite[Proposition~2.11]{lucy} is correct, which is as follows:

\begin{Theorem}\label{t7}\cite[Proposition~2.11]{lucy}
Let $R\subset T$ be a minimal ring extension and $J$ be the crucial maximal ideal. Then $(R:T) \in \text{Spec}(R)$. Moreover, if no maximal ideal in $T$ lies over $J$, then $(R:T)\subset J$, $T_{J} = R_{(R:T)}$ is local, $(R_J:T_{J}) = (R:T)R_J$ is the maximal ideal in $T_{J}$, and $\text{height}(J/(R:T)) = 1$.
\end{Theorem}
 
We give one more example to counter $(2)$. More precisely, the next example shows that if $R\subset T$ is a minimal ring extension and $T$ is an integral domain with $(R:T) = 0$, then degree of $T$ over $R$ may not be prime.

\begin{Example}\label{e2}
Let $n\geq 4$. Then there exist field extension $K$ of $\mathbb Q$ such that $Gal_{\mathbb Q}(K) = S_n$. In fact, choose $f(X)\in \mathbb Q[X]$ irreducible of degree $4$ such that $\left|Gal_{\mathbb Q}(K)\right| = 24$. Let $\alpha$ be a root of $f(X)$. Then $dim_{\mathbb Q} {\mathbb Q(\alpha)} = 4$ and $\mathbb Q\subset \mathbb Q(\alpha)$ does not have any intermediate ring.
\end{Example}

We now present the correct and modified version of the countered result $(2)$. Note that if $R$ is local, then $R_J = R$. Thus, our next theorem shows that $(2)$ is correct only if $R$ is local.

\begin{Theorem}\label{t8}
If $R\subset T$ is a minimal ring extension and $T$ is an integral domain, then $(R:T) = 0$ if and only if $R$ is a field and $T$ is a minimal field extension of $R$, or $R_J$ is a valuation ring of altitude one and $T_{J}$ is its quotient field.
\end{Theorem}

\begin{proof}
Suppose first that $(R:T) = 0$. If $J = (0)$, then $R$ is a field. Since $T$ is a minimal ring extension of $R$ and $T$ is an integral domain, $T$ is a minimal field extension of $R$, by \cite[Lemme~1.2]{ferrand}. If $J\neq (0)$, then $(R:T)\subset J$. By \cite[Theorem~2.13]{lucy}, we have $R_0 = R_{(R:T)} = T_{J}$. Therefore, $T_{J}$ is the quotient field of $R$. Also by \cite[Th\'{e}or\`{e}me~2.2]{ferrand}, $R_J\subset T_{J}$ is a minimal ring extension, that is, $R_J$ is a maximal proper subring of $T_{J}$. Therefore, $R_J$ is a valuation ring of altitude one by \cite[Proposition~6, VI, 4.5]{bourbaki} and $T_{J}$ is its quotient field. Conversely, if $R$ is a field, then clearly $(R:T) = 0$. Also, if $R_J$ is a valuation ring of altitude one and $T_{J}$ is its quotient field, then $(R_J:T_{J}) = 0$. Note that $T$ cannot be integral over $R$ as $R_J$ is integrally closed. Since $R\subset T$ is a minimal ring extension, $R$ is integrally closed in $T$. Now, by \cite[Th\'{e}or\`{e}me~2.2(ii)]{ferrand}, there is no maximal ideal in $T$ lies over $J$. Therefore, by Theorem \ref{t7}, $(R_J:T_{J}) = (R:T)R_J$. Hence, $(R:T) = 0$.
\end{proof}

The next proposition is a modified version of the result $(3)$. Note that if $R$ is local, then $R_J = R$. Thus, our next proposition shows that $(3)$ is correct only if $R$ is local. 

\begin{Proposition}\label{p1}
Let $f:R\hookrightarrow T$ be a minimal ring homomorphism. If $f:R\hookrightarrow T$ is a flat epimorphism, then $R_J/(R_J:T_{J})$ is a one-dimensional local domain and $T_{J} = R_{(R:T)}$.
\end{Proposition}

\begin{proof}Since $R \subset T$ is a minimal ring extension, $T_{J}$ is a minimal ring extension of $R_{J}$, by \cite[Th\'{e}or\`{e}me~2.2]{ferrand}. Now, if $f:R\hookrightarrow T$ is a flat epimorphism, then by \cite[Th\'{e}or\`{e}me~2.2(ii)]{ferrand}, there is no maximal ideal in $T$ lies over $J$. Thus, by Theorem \ref{t7}, $(R_J:T_{J}) \in \text{Spec}(R_{J})$, $T_{J} = R_{(R:T)}$, $(R_J:T_{J}) = (R:T)R_J$, and $(R:T)\subset J$. This gives $(R_J:T_{J})\subset JR_{J}$. Since $R_J$ is a local ring, $JR_J$ is the crucial maximal ideal of the minimal ring extension $R_J\subset T_J$. Hence, by \cite[Th\'{e}or\`{e}me~2.2(ii)]{ferrand}, there is no maximal ideal in $T_{J}$ lies over $JR_{J}$. Now, by Theorem \ref{t7}, $\text{height}(JR_{J}/(R_J:T_{J})) = 1$. Therefore, $R_J/(R_J:T_{J})$ is a one-dimensional local domain.
\end{proof}

Our last theorem is a modified version of the result $(4)$. Note that if $R$ is local, then $R_J = R$. Thus, our next theorem shows that $(4)$ is correct only if $R$ is local.

\begin{Theorem}\label{t9}
Let $R\hookrightarrow T$ be a minimal ring homomorphism. Then $R\hookrightarrow T$ is a flat epimorphism if and only if $R_J/(R_J:T_{J})$ is a one-dimensional valuation ring and $T_{J}/(R_J:T_{J})$ is its quotient field.
\end{Theorem}

\begin{proof}
Let $R\hookrightarrow T$ be a flat epimorphism. Then by \cite[Th\'{e}or\`{e}me~2.2]{ferrand}, $R_J\subset T_J$ is a minimal ring extension and there is no maximal ideal in $T$ lies over $J$. Therefore, Theorem \ref{t7} yields that $T_{J}/(R_J:T_{J})$ is a field. Now, by \cite[Theorem~2.7]{lucy}, $R_J/(R_J:T_{J}) \subset T_{J}/(R_J:T_{J})$ is a minimal ring extension. Since $(R_{J}/(R_J:T_{J}):~T_{J}/(R_J:T_{J})) = 0$, we conclude that $R_J/(R_J:T_{J})$ is a one-dimensional valuation ring and $T_{J}/(R_J:T_{J})$ is its quotient field, by Theorem \ref{t8} and Proposition \ref{p1}. Conversely, assume that $R_J/(R_J:T_{J})$ is a one-dimensional valuation ring and $T_{J}/(R_J:T_{J})$ is its quotient field. Then $R_J/(R_J:T_{J})$ is integrally closed in $T_{J}/(R_J:T_{J})$. Thus, Ferrand's dichotomy \cite[Th\'{e}or\`{e}me~2.2]{ferrand} gives that $R_J/(R_J:T_{J}) \hookrightarrow T_{J}/(R_J:T_{J})$ is a flat epimorphism. Therefore, $R_J\hookrightarrow T_{J}$ is a flat epimorphism and hence so is $R\hookrightarrow T$.
\end{proof}
 
\begin{Remark}
There is an error in the proof of \cite[Theorem~3.7]{dobbs3}. Note that $R$ is not local in \cite[Theorem~3.7]{dobbs3} but the proof of \cite[Theorem~3.7]{dobbs3} is citing \cite[Proposition~3.5]{glaz} which is true for local rings only. The error in the proof arises because the authors used \cite[Proposition~3.5]{glaz} to prove that $(R/P)_{M/P}$ is a valuation domain in $(1)\Rightarrow (3)$. But as we have seen earlier, \cite[Proposition~3.5]{glaz} is valid for local rings only. Thus, the proof of \cite[Theorem~3.7]{dobbs3} is not correct. Note that in \cite[Theorem~3.7]{dobbs3}, we have $(R/P)_{M/P} \cong R_M/PR_M$ where $P = (R:T)$ and $M$ is the crucial maximal ideal of the minimal ring extension $R\subset T$. Thus, by Theorem $\ref{t9}$, $(R/P)_{M/P}$ is a valuation domain and hence \cite[Theorem~3.7]{dobbs3} holds.
\end{Remark}

\section*{\bf Acknowledgments}
The authors thank Prof. Alok K. Maloo for providing Example \ref{e2} and for 
the stimulating discussions. The first author was supported by a grant from UGC India, Sr.
No. 2061440976. The second author was supported by the MATRICS grant from DST-SERB, No. MTR/2018/000707.


\begin{thebibliography}{00000}
\bibitem{azarang1} A. Azarang, On maximal subrings, Far East J. Math. Sci. (FJMS) 32 (2009), no. 1, 107-118.

\bibitem{bourbaki} N. Bourbaki, Commutative algebra, Addison-Wesley Publishing Co., Reading, Mass., 1972.

\bibitem{dobbs1} D. E. Dobbs, Every commutative ring has a minimal ring extension, Comm. Algebra 34 (2006) 3875-3881.

\bibitem{dobbs2} D. E. Dobbs and J. Shapiro, A classification of the minimal ring extensions of an integral domain, J. Algebra 305 (2006) 185-193.

\bibitem{dobbs3} D. E. Dobbs and J. Shapiro, A classification of the minimal ring extensions of certain commutative rings, J. Algebra 308 (2007) 800-821.

\bibitem{dobbs4} D. E. Dobbs, Recent progress on minimal ring extensions and related concepts, Commutative Rings: New Research, Nova Science, New York, 2009 pp. 13-38.

\bibitem{lucy} L.I. Dech\'{e}ne, Adjacent extensions of rings, PhD dissertation, University of California, Riverside, 1978.

\bibitem{ferrand} D. Ferrand and J.-P. Olivier, Homomorphismes minimaux d'anneaux, J. Algebra 16 (1970) 461-471.

\bibitem{nagata} M. Nagata, Local rings, Interscience, New York, 1962.

\bibitem{glaz} G. Picavet and M. Picavet-L' Hermitte, About minimal morphisms, in: Multiplicative Ideal Theory: A Tribute to the Work of Robert Gilmer, Springer-Verlag, New York, 2006 pp. 369-386.

\bibitem{glaz1} G. Picavet and M. Picavet-L' Hermitte, Quasi-Pr\"ufer extensions of rings, in: Fontana M., Frisch S., Glaz S., Tartarone F., Zanardo P. (eds) Rings, Polynomials, and Modules, Springer, Cham, 2017 pp. 307-336.



\end{thebibliography}

\end{document}